\documentclass[12pt,amsfonts, epsfig]{amsart}

\usepackage{amsmath, amscd, amssymb}
\usepackage[frame,cmtip,arrow,matrix,line,graph,curve]{xy}
\usepackage{epsfig}
\usepackage{graphpap, color}
\usepackage[mathscr]{eucal}
\usepackage{mathrsfs}

\usepackage{pstricks}
\usepackage{color}
\usepackage{cancel}

\numberwithin{equation}{section}

\newcommand{\PP}{\mathbb{P}}

\newcommand{\bL}{\mathbf{L}}

\newcommand{\bk}{\mathbf{k}}

\newcommand{\kk}{\bk}

\newcommand{\bP}{\mathbf{P}}

\newcommand{\cal}{\mathcal}

\def\cA{{\cal A}}

\def\cX{{\cal X}}

\def\Eb{E^\bullet}

\def\sE{{\mathscr E}}
\def\sEb{\sE^\bullet}
\def\sF{{\mathscr F}}
\def\sFb{\sF^\bullet}
\def\sG{{\mathscr G}}
\def\sH{{\mathscr H}}
\def\sI{{\mathscr I}}
\def\sJ{{\mathscr J}}
\def\sL{{\mathscr L}}

\def\sO{{\mathscr O}}

\def\fM{\mathfrak{M}}
\def\2M{M}
\def\wfM{\widetilde{\2M}}

\def\ff{\mathfrak{f}}

\def\sta{^\ast}

\def\sta{^{\ast}}

\def\sta{^*}

\def\lra{\longrightarrow}

\def\lsta{_{\ast}}

\newcommand{\Ga}{\Gamma}

\newcommand{\ee}{{\bf e}}

\def\begeq{\begin{equation}}
\def\endeq{\end{equation}}
\def\and{\quad{\rm and}\quad}

\def\defeq{:=}

\def\sub{\subset}

\def\and{\quad\text{and}\quad}

  \DeclareMathOperator{\Hom}{Hom}

\DeclareMathOperator{\rank}{rank}

\let\lab=\label

\newtheorem{prop}{Proposition}[section]
\newtheorem{theo}[prop]{Theorem}

\newtheorem{coro}[prop]{Corollary}
\newtheorem{defi}[prop]{Definition}
\newtheorem{conj}[prop]{Conjecture}

\theoremstyle{definition}
\newtheorem{say}[prop]{}
\newtheorem{rema}[prop]{Remark}

\DeclareMathOperator{\coker}{coker}

\def\Po{{\mathbb P^1}}

\def\Pf{{\mathbb P}^4}
\def\Pn{{\mathbb P}^n}
\def\PP{{\mathbb P}}

\def\bP{{\bf P}}

\def\fM{\mathfrak{M}}

\def\fMPdgn{\overline{\fM}_{g,n}(\PP^m,d)}

\def\wX{\widetilde{\cX}}

\def\sta{^\ast}

\let\lab=\label

\def\lab#1{\label{#1}[{#1}]\  }

\def\lab{\label} 

\def\beq{\begin{equation}}
\def\eeq{\end{equation}}

\def\lab{\label}

\title[Primary Enumerative Invariants]
{Derived Resolution Property for Stacks, Euler Classes and
Applications}

\author{Yi Hu}
\address{Department of Mathematics, University of Arizona, USA.}
\email{yhu@math.arizona.edu}
\author{Jun Li}
\address{Department of Mathematics, Stanford University, USA.}
\email{jli@math.stanford.edu}  
\date{}

\begin{document}
\maketitle

\begin{abstract}
By resolving any perfect derived object over a
Deligne-Mumford stack, we define its Euler class.
We then apply it to define the Euler numbers for a smooth Calabi-Yau threefold in $\Pf$.
These numbers are conjectured to be  the reduced Gromov-Witten invariants
and to determine the usual Gromov-Witten numbers of the smooth
quintic as speculated by J. Li and A. Zinger.
\end{abstract}


\section{Introduction}

Let $\overline{\fM}_g(\bP,d)$ be the DM stack of degree $d$ genus-$g$ stable maps to a projective space $\bP$. We let
\beq\begin{CD} { \cX} @>{\pi}>> \overline{\fM}_g(\bP,d) \and { \cX} @>{\ff}>>
\bP
\end{CD}\eeq
be its universal family. For any integer
$k>0$,  the derived object $R \pi_* (\ff^*
\sO_{{\PP^m}}(k))$ is quasi-isomorphic to a complex of locally free sheaves
$$\begin{CD} \sE^\bullet=[\sE_0 @>{\varphi}>> \sE_1].
\end{CD}$$
The main purpose of this article is to define the Euler class
$$\ee(R \pi_*( \ff^*\sO_{{{\bP}}}(k))$$
of this complex
over the primary component of $\overline{\fM}_g(\bP,d)$.

In fact, we consider any perfect derived object $\sEb$ in the
 bounded derived category  ${\rm D}^b(M)$ of an
integral DM stack $M$ with  cohomologies concentrated in the non-negative
places. Our main theorem
says that any such a perfect derived object $\sEb$, such as $R \pi_* (f^* \sO_{{\PP^m}}(k))$
in the above, can be resolved to have locally free sheaf cohomology
$\sH^0$ after birational base change.

\begin{theo}\lab{resolution} {\rm (Existence of Resolution.)}
Let $\sEb$ be any perfect derived object over an integral DM
stack $M$.  Assume that $\sEb$ can  locally be represented 
by a complex of locally free sheaves of finite length supported
only in non-negative degrees. 
Then there is another
integral DM stack $\wfM$ and a surjective birational
morphism $f: \wfM\to \2M$ such that  $\sH^0(\bL f\sta \sEb)$ is
locally free.
\end{theo}

Moreover, among the above stacks $\wfM$, there is
also a unique one (up to isomorphism) that is  minimal (see Theorem
\ref{resolution-restated} and Proposition \ref{stack:universality}
for the precise statements). In the main text, we will prove the stronger Theorem
\ref{resolution-restated} from which the above  is a consequence.
The main technical theorem for this purpose is Theorem
\ref{determ}.

\begin{defi}
Suppose further that the cohomologies in positive places
$\sH^{i>0}(\sEb)$ are all torsion sheaves  over $M$.
Then, we define the Euler class ${ \ee}(\sEb)$ in the Chow group $A_* M$  of cycles on
$\2M$ by, \beq { \ee}(\sEb) := f_*(c_r(\sH^0(L f\sta \sEb)) \cdot
[\wfM]),\eeq  where $r=\rank  \sH^0(L f\sta \sEb)$.
\end{defi}

\begin{prop}  Let $M$ be an integral DM stack 
and $\sEb$ a perfect derived object as in Theorem
\ref{resolution}. Suppose $\sH^{i<0}(\sEb)=0$
 and $\sH^{i>0}(\sEb)$ are torsion sheaves. Then
the Euler class ${ \ee}(\sEb)$ is independent of the choice of the
resolutions $f: \wfM\to \2M$.
\end{prop}

For more precise statements, see Proposition \ref{independent} and Corollary \ref{Euler}.

\medskip

{
The above, when applied to the derived object
$R \pi_* \ff^* \sO_{\PP^4}(5)$ restricted to
the primary component $\overline{\fM}_g(\PP^4, d)'$
 of the moduli stack
$\overline{\fM}_g(\PP^4,d)$, enables us to construct the modular Euler class when $d > 2g-2$.
(The general points of $\overline{\fM}_g(\PP^4,d)'$ are
maps with smooth domains, and
$\overline{\fM}_g(\PP^4,d)'$ is irreducible and of the expected dimension when
$d >2g-2$. cf. \ref{primaryComp}.)  In this case, letting $\ff'$ be the restriction of the universal family $\ff$
to $\overline{\fM}_g(\PP^4,d)'$, then $R^1 \pi_* \ff^{\prime\ast} \sO_{\PP^4}(5)$
is a torsion sheaf over $\overline{\fM}_g(\PP^4,d)'$.

\begin{defi} For $d> 2g-2$, we define the modular Euler class of $R \pi_* \ff^* \sO_{\PP^4}(5)$
over $\overline{\fM}_g(\PP^4,d)'$ to be
\beq { \ee}(R \pi_* \ff^{\prime\ast}
\sO_{\PP^4}(5)) \in A_* (\overline{\fM}_g(\PP^4,d)');
\eeq
for any smooth Calabi-Yau manifold $Q$ in $\PP^4$, we define
\beq N_{g,d}'(Q) = \deg \ee(R \pi_* \ff^{\prime\ast} \sO_{\PP^4}(5)).
\eeq
\end{defi}

We believe that these numbers $N_{g,d}'(Q)$ are the reduced GW-invariants speculated by Li and Zinger:

\begin{conj}\lab{conj}  Let $N_{g,d}(Q)$ be the genus $g$, degree $d$ GW-invariants of the smooth
quintic $Q\sub\PP^4$.
There are universal constants $c_h$
such that for  $d > 2g-2$,

$$N_{g,d}(Q)= \sum_{0 \le h \le g} c_h N_{h,d}'(Q).$$
\end{conj}
}

\medskip
While this paper being prepared, the first author was partially supported by NSA grant MSP07G-112  and NSF DMS 0901136;
the second author was partially supported by NSF DMS-0601002.

\section{Conventions and Terminology}
\lab{pre}


\begin{say}
Throughout the paper, we fix an arbitrary algebraically closed
base field $\kk$ of characteristic 0.  All schemes and stacks in this paper
are assumed to be noetherian over $\kk$.
\end{say}

\begin{say}
Paragraphs are enumerated, so are equations. For instance, 3.2
refers to Paragraph 3.2, while (3.2) refers to Equation (3.2). \S
3.2 refers to a subsection.
\end{say}

\begin{say}\lab{smoothness} Unless otherwise stated,
smoothness is in the sense of DM stack or Artin stack. On the
coarse moduli level, this roughly means that the moduli space  is
locally in the \'etale topology a quotient of a smooth variety by
a finite group or in the topology of smooth morphisms a quotient
of a smooth variety by a group scheme.
\end{say}

\begin{say} All morphisms between stacks are assumed to be
representable.
\end{say}

\begin{say}
Let $X$ be a scheme, $D$ a Cartier divisor of $X$ and $Z$ a closed
subscheme of $X$. We will write $\sO_Z(D)$ for the restriction
$\sO_X(D)|_Z$.
\end{say}

\begin{say}
For any right exact functor $F$ such as $f^*$ from an Abelian
category $\mathscr A$ to another $\mathscr B$, we use $LF$ to
denote the left derived functor. Similarly, for a left exact
functor $F$ such as $f_*$, $R F$ is the right derived
functor.
\end{say}

\section{Diagonalizing Sheaf Homomorphism}
\lab{determinantalBlowup}

\begin{say}
Let $\2M$ be a DM-stack. We denote by ${\rm D}^b(\2M)$ the derived
category of bounded complexes of coherent sheaves over $\2M$.  An
object $\sEb\in {\rm D}^b(\2M)$ is called {\it perfect} if locally
it can be represented by a complex of locally free sheaves of
finite length. Equivalently, the stack $\2M$
admits an \'etale cover $U$ by a scheme such that
there is a finite length complex $\sF^\bullet$ of locally
free sheaves over $U$ such that $\sEb$ is represented by
$\sF^\bullet$ in ${\rm D}^b(U)$.
\end{say}

\begin{say}
Let { an integral} DM stack $\2M$ and { a} derived object $\sEb$ { be}
as in Theorem \ref{resolution}. Since $\sEb$ is perfect, we can cover $\2M$ by
open charts $\coprod U$ and for each open subset $U$ there is a
finite length complex $\sF^\bullet$ of locally free sheaves over
$U$ such that $\sEb|_U$ is represented by $\sF^\bullet$.  By our
assumption, we may assume that $\sF^\bullet$ has the form
\beq\begin{CD} \sF_0 @>{\psi}>> \sF_1 @>>> \cdots
\end{CD}
\eeq Our aim is to resolve the sheaf $\sH^0(\2M,\sEb)$. Note that
we have
$$\sH^0(U,\sEb|_U) \cong \ker \psi.$$
\end{say}

\begin{say} To  resolve the sheaf $\sH^0(\2M,\sEb)$, we
consider the following { model}.  Let $X$ be { a} scheme, { not necessarily
irreducible or reduced,} and $$\varphi: E \lra
F$$ be a homomorphism between locally free sheaves over $X$. We let
$\bigwedge^i \varphi: \bigwedge^i E \lra \bigwedge^i F$ be the
induced homomorphism between the wedge products, $i \ge 0$. We
also view $\bigwedge^i \varphi$ as a section of $\Hom(\bigwedge^i
E, \bigwedge^i F)$. Denote by $m$ the rank of the image sheaf
${\rm Im} \varphi$. This is the smallest integer such that
$\bigwedge^{m+1} \varphi \equiv 0$.
\end{say}

\begin{defi}\lab{defi:determinantal} For any  $0 \le r \le m-1$, we let
$Z_{\varphi,r} \subset X$ { be} the subscheme of { vanishing} of
the section $\bigwedge^{r+1} \varphi$ and call it the
$r$-determinantal subscheme of $X$ with respect to $\varphi$. { Its scheme
structure  is  given}  by the determinantal ideal $\sI_{\varphi, r}$ { which is}
generated by all $(r+1) \times (r+1)$ minor determinants of any
local matrix representation of $\varphi$.
\end{defi}

Note that the ideal sheaf $\sI_{\varphi, r}$ does not depend on
the choice of local matrix representation; it is supported on the
locus of points $w$ such that $\dim {\rm Im} (\varphi (w)) \le r$.

\begin{say}\lab{detIdealBaseChange}
Determinantal ideals have base change property. If $f: Y \lra X$
is a morphism between schemes, then the $r$-determinantal ideal of
$f^*\varphi: f^*E \lra f^*F$ is the pullback of the
$r$-determinantal ideal of $\varphi: E \lra F$. That is,
$\sI_{f^*\varphi, r} = f^*\sI_{\varphi, r}$.
\end{say}


\def\wX{\widetilde{X}}

\begin{say}
Now, we will define inductively the blowing ups of $X$ along the
determinantal ideal sheaves\footnote{We note here that similar
idea has been used in \cite{HLS}.}.  First, we let
$$\begin{CD} b_0: X_0 @>>> X \end{CD}$$
be the blowing-up of $X$ along the ideal sheaf $\sI_{\varphi,0}$.
For any $0 \le r \le m-2$, assume that $$\begin{CD} b_r: X_r @>>>
X_{r-1} \end{CD}$$  is already defined. We let
$$\begin{CD} \varphi_r: E_r @>>> F_r \end{CD}$$ be the pullback of $\varphi: E \lra
F$. Then we define
$$\begin{CD} b_r: X_{r+1} @>>>
X_{r} \end{CD}$$  to be the blowup of $X_r$ along the ideal sheaf
$\sI_{\varphi_r,r+1}$. After the final blowing up
$$\begin{CD} b_{m-1}: X_{m-1} @>>>
X_{m-2}, \end{CD}$$ we define
$$\begin{CD} b: \wX:=X_{m-1} @>>>
X \end{CD}$$ to be the induced iterated blowing up. We also let
$D_r$ to denote the exceptional divisor of the birational morphism
$b_r: X_r \lra X_{r-1}$.
\end{say}

\begin{theo}\lab{resolve-kernel} The birational morphism $b: \wX
\lra X$ resolves the kernel sheaf of $\varphi: E \lra F$. That is,
{ the kernel sheaf $\ker (b^* \varphi|_{\widetilde{X}'})$ is
locally free where $\widetilde{X'}$  is { any} irreducible component
of $\widetilde{X}$ with the reduced scheme structure.}
\end{theo}

To prove this theorem,  we will state and prove the technical but
stronger Theorem \ref{determ}, from which the above is an
immediate consequence. { To this end, we need the pivotal notion of locally
diagonalizable homomorphism.}

\begin{defi}\lab{diagonalizable} A homomorphism $\varphi: \sO_X^{\oplus p} \lra \sO_X^{\oplus
q}$ is said to be diagonalizable if we have direct sum
decompositions by trivial sheaves \beq\lab{decompositionE-F}
\sO_X^{\oplus p}=G_0 \oplus \bigoplus_{i=1}^l G_i \and
\sO_X^{\oplus q}=H_0 \oplus \bigoplus_{i=1}^l H_i \eeq with
$\varphi (G_i) \subset H_i$ for all $i$ such that
\begin{enumerate}
\item $\varphi|_{G_0}=0$; \item for  every $1\le i \le l$,
$\varphi|_{G_i}$ equals to $p_i I_i$ for some $0 \ne p_i \in
\Gamma(\sO_X)$  where $I_i: G_i \to H_i$ is an isomorphism; \item
{ $\langle p_i \rangle   \supsetneqq \langle p_{i+1} \rangle$}
\end{enumerate}
\end{defi}

\begin{defi} A homomorphism $\varphi: E \lra F$ between locally
free sheaves of a scheme $X$ is  locally diagonalizable if there
are trivializations of $E$ and $F$ over { some} open covering of $X$
such that  $\varphi: E \lra F$ is diagonalizable over every open
subset.
\end{defi}

It is routine to check the following { useful observations}.

\begin{prop}\lab{usefulFacts}
Suppose that a homomorphism $\varphi: E \lra F$ is  locally
diagonalizable, then
\begin{enumerate}
\item for every $0 \le r \le m-1$, the  determinantal ideal
$\sI_{\varphi, r}$ is invertible; \item  for every irreducible
component $X'$ of $X$ with the reduced scheme structure, $\ker
(\varphi|_{X'})$ is locally free\footnote{Note here that the rank
of $\ker \varphi$ depends on the properties of the functions
$p_i$, hence may not be constant over $X$. { Further, when
$X'$ is not reduced, $\ker (\varphi|_{X'})$ needs not to be
locally free. These technical issues lead us to use the
$``$integral$"$ assumption whenever and only when we want to
produce a locally free sheaf.}}; \item if $f: Y \lra X$ is a
morphism, then $f^* \varphi$ is also locally diagonalizable (i.e.,
$``$locally diagonalizable$"$ has base change property).
\end{enumerate}
\end{prop}


We now arrive at our main technical theorem.

\begin{theo}\lab{determ} {{\rm (Diagonalization.)}}
For any $0 \le r \le m-1$, there is an open covering of $X_r$
trivializing $E_r$ and $F_r$ such that over each open subset $U$
we have direct sum decompositions by trivial sheaves $$ E_r=
\bigoplus_{i=0}^r G_i \oplus G_{r+1} \and F_r= \bigoplus_{i=0}^r
H_i \oplus H_{r+1} $$ with $\varphi_r (G_i) \subset H_i$ for all
$0 \le i \le r$ and $\varphi_r (G_{r+1}) \subset H_{r+1}$ making
the following true
\begin{enumerate} 
{
 \item  $G_i \cong H_i\cong \sO_U$ for all $0 \le i \le r$, hence $\bigoplus_{i=0}^r G_i \cong \bigoplus_{i=0}^r
H_i  \cong \sO_U^{\oplus(r+1)}$; \item for  every $0\le i \le r$,
$\varphi_r|_{G_i}$ equals to $p_i I_i$ for some $0 \ne p_i \in
\Gamma(\sO_X)$ where $I_i: G_i \to H_i$ is an isomorphism; \item
$p_i | p_{i+1}$ $(0 \le i \le r-1)$; \item $p_r$ divides
$\varphi_r|_{G_{r+1}}$. }
\end{enumerate}
 In particular, when $r=m-1$, we obtain that the homomorphism $b^* \varphi$ is locally diagonalizable,
 and { hence the kernel sheaf
  $\ker (b^* \varphi|_{\widetilde{X}'})$ is locally free where
$\widetilde{X}'$  is an irreducible component of $\widetilde{X}$
with the reduced scheme structure.}
\end{theo}
\begin{proof} We will prove by induction on $r$.

First, consider the case of $r=0$. Take any point $\xi \in X$.
 Locally around $\xi$, we can trivialize
$E$ and $F$ over an open neighborhood $U$ of $\xi$ and choose
bases of $E$ and $F$ such that $\varphi$ is given by the matrix
$(\mu_{ij})$. 

If $Z_{\varphi,0} = \emptyset$, or equivalently $\sI_{\varphi,
0}=\sO_U$, then there is a $\mu_{ij}$ such that $\mu_{ij}(0) \ne
0$. By shrinking $U$ we may assume that $\mu_{ij} \in \Gamma
(\sO_U^*)$. This way, by a basis change, we can arrange
decompositions
$$E \cong G_1\oplus G_2 \and F \cong H_1\oplus H_2$$
such that $G_1, H_1 \cong  \sO_U$, $\varphi (G_1) \subset H_1,
\varphi (G_2)\subset H_2$,
 and $\varphi|_{G_1}: G_1 \lra H_1$ is an isomorphism.
Since $X_0=X$ in this case, the statements of the theorem hold.

If $Z_{\varphi,0}$ is a Cartier divisor, i.e., the ideal
$\sI_{\varphi,0}$ is the principal ideal $\langle p \rangle$
generated by some $p \in \Gamma(\sO_U)$, then we can write
$(\mu_{ij})$ as $(p \cdot\nu_{ij})$. Since  $\langle \mu_{ij}
\rangle = \langle p \rangle$, we see that $\langle \nu_{ij}
\rangle = \sO_U$. This implies that the homomorphism $\varphi$
factors as
$$ 
\begin{CD} \varphi: E @>{\varphi'}>> F(-Z_{\varphi,0}) @>{\rm inclusion}>>
F,
\end{CD}
$$  and $\varphi'=(\nu_{ij})$ has that $I_{\varphi_0',0} = \sO_U$. Now apply the previous case
to the homomorphism $\varphi'$, one checks that we  obtain
decompositions
$$E=G_1\oplus G_2 \and F=H_1\oplus H_2$$
such that $G_1, H_1 \cong  \sO_U$, $\varphi (G_1) \subset H_1,
\varphi (G_2)\subset H_2$,
 and $\varphi|_{G_1}: G_1 \lra H_1$ equals $pI_1$ where
$I_1: G_1 \lra H_1$ is an ismomorphism
 and $p$ divides  $\varphi|_{G_2}$.
Again, since $X_0=X$ in this case, the statements of the theorem hold.

If $Z_{\varphi,0}$ is not a Cartier divisor, that is,
$\sI_{\varphi,0}$ is not principal, we blow up $Z_{\varphi,0}$ to
obtain
$$b_0: X_0 \lra X \and b_0^* \varphi: b_0^* E \lra b_0^* F. $$
 By
\ref{detIdealBaseChange}, we have that $\sI_{b_0^*\varphi,0}=b_0^*
\sI_{\varphi,0}$. Since $b_0^* \sI_{\varphi,0}$ is principal, this
reduces to the previous case. Note that $b_0^*\varphi =
\varphi_0$. Thus, the case 0 is proved.

Assume now that the assertion holds for $r$. Since the question is
local, we can restrict our focus on an affine open subset $U$ of
$X_r$ such that $E_r$ are $F_r$  are all trivialized with the
desired decomposition as in the theorem and $\varphi_r: E_r \lra
F_r$ has the desired properties  as in the theorem. { In terms of
the suitable bases of $E_r$ and $F_r$ as in the theorem granted by
the inductive assumption,}  all these mean is that we can
represent $\varphi_r$ by the diagonal matrix
 $${\rm diag}[p_0I_0, \cdots, p_rI_r, B] $$
 where $B$ is the matrix representation of $\varphi_r|_{G_{r+1}}$
 and $p_r | B$. { Since $p_r | B$, we may write $B= p_r B'$ and let
 $ \varphi_r': G_{r+1} \lra H_{r+1}$ be the homomorphism
 corresponding to $B'$.
From the above representation, we see that \beq\lab{differByL}
\sI_{\varphi_r, r+1} =(p_0 \cdots p_r) p_r \sI_{\varphi_r', 0}.
\eeq }
Hence blowing up
$\sI_{\varphi_r, r+1}$ is the same as blowing up $\sI_{\varphi_r',
0}$. Now we can apply the case 0 to the homomorphism
$$\varphi_r': G_{r+1} \lra H_{r+1}.$$
From here, it is routine to check against the three cases of case 0 (for $\varphi_r'$)
so that we will obtain the desired decompositions for $E_{r+1}$ and
 $F_{r+1}$ with the desired properties for $\varphi_{r+1}$.

By induction, this proves the statements (1)--(4) of the theorem.

To finish off, in the case of $r=m-1$, because $\oplus_{i=0}^r G_i
\cong \sO_{ U}^{\oplus m}$ with $m$ the maximal rank of
$\varphi_{m-1}=b^*\varphi$, we must have that $\varphi_{m-1}
(G_{r+1})=0$. This implies that $b^*\varphi$ is { locally
diagonalizable, and in particular, if we restrict $b^*\varphi$ to
an irreducible component with the reduced structure, its kernel
sheaf is locally free.}

This completes the proof.
\end{proof}

\begin{rema}The requirement that each $G_i$ is isomorphic to
$\sO_U$ in the above theorem appears to be $``$stronger$"$ than
Definition \ref{diagonalizable}, but in fact, they are equivalent
{ since in this theorem  we allow $\langle p_i \rangle = \langle
p_{i+1} \rangle$ for some pairs $p_i$ and $p_{i+1}$.}
\end{rema}

\begin{prop}\lab{prop:universality} {\rm (Universality.)}
Let $\varphi: E \lra F$ be a homomorphism between locally free
sheaves  over a scheme $X$ and $\widetilde{X}$ the blowup of $X$
along determinantal loci of $\varphi$. If $f: Z \lra X$ is any
dominant morphism between  schemes such that the pullback
homomorphism $f^*\varphi: f^*E \lra f^*F$  is locally
diagonalizable, then $f$ factors uniquely through $\widetilde{X}$.
\end{prop}
\proof Let $m$ be the rank of the image sheaf ${\rm Im} \varphi$.
For any  $0 \le r \le m-1$, we will show inductively that $f$
factors uniquely through $X_r$.

When $r=0$, consider any point $\xi \in Z$ such that $f(\xi)$
belongs to the $0$-determinantal locus of $\varphi$.  As in the
proof of Theorem \ref{determ}, locally around $f(\xi) \in W$, we
can represent $\varphi$ by a matrix $(\mu_{ij})$ with $\mu_{ij}
\in \Gamma(\sO_U)$, where $U$ is an open neighborhood of $f(\xi)$.
The homomorphism $f^*\varphi: f^*E \lra f^*F$ is given by the
pullbacks $(f^*\mu_{ij})$.  Let $\sJ=f^{-1} \sI_{\varphi,0}$ be
the ideal generated by $(f^*\mu_{ij})$. Since $f$ is dominant,
$f(Z)$ is not entirely contained in the 0-determinantal locus of
$\varphi$, hence $\sJ \ne 0$. But then, since $f^*\varphi: f^*E
\lra f^*F$ is diagonalizable, we can represent $f^*\varphi$ as a
diagonal matrix ${\rm diag}[p_1I_1, \cdots, p_lI_l, 0]$. This
shows that
 $\sJ$ is principal. By the universal property of blowing up
(Proposition 7.14 of \cite{Hartshorne}), $f$ factors uniquely
through  $X_0$.

Now assume that the claim holds for $r$: $f$ factors uniquely
through $X_r$
$$\begin{CD} f: Z @>{g}>> X_r  @>{b_r}>>X.\end{CD}$$
Since $X_{r+1}$ is the blowup of $X_r$ along $\sI_{\varphi_r',0}$
(here we use the notation as in the proof Theorem \ref{determ}),
we need only to consider a small open neighborhood of any  point
$\xi \in Z$ such that $g(\xi)$ belongs to the $0$-determinantal
locus of $\varphi_r'$. Since $f$ is dominant, by the base change
property, we have
$$g^{-1} \sI_{\varphi_r',0}=\sI_{g^*\varphi_r',0} \and
f^{-1}\sI_{\varphi, r+1}= \sI_{f^*\varphi, r+1}.$$
 By \eqref{differByL} in the proof of Theorem \ref{determ}, we see
 that
$\sI_{g^*\varphi_r',0}$ and $\sI_{f^*\varphi, r+1}$ differ by an
invertible sheaf. Since $f^*\varphi: f^*E \lra f^*F$ is
diagonalizable, from its diagonal representation we conclude that
$\sI_{f^*\varphi, r+1}$, { which is} $f^{-1}\sI_{\varphi, r+1}$,
is invertible. Hence, so is $\sI_{g^*\varphi_r',0} = g^{-1}
\sI_{\varphi_r',0}$.
 This implies that $g$
 factors uniquely through $X_{r+1}$, hence so does $f$.

 By induction, this finishes the proof.
\endproof

\begin{defi} For any homomorphism $\psi: E \lra F$ between two locally free sheaves over a scheme $X$,
we say that a point $x \in X$ is $\psi$-regular if $\psi (x)$ is
of maximal rank.
\end{defi}

It is easy to see that if $x \in X$ is $\psi$-regular, then
locally around $x$,  $\psi: E \lra F$ is diagonalizable (and in
fact, can be  diagonalized to the form ${\rm diag}(I, 0)$).

\begin{rema} The universality proposition \ref{prop:universality}
may also hold for { any} morphism $f: Z \lra X$ such that $f(Z)$ contains
$\varphi$-regular points and $f^* \varphi$ is locally
diagonalizable. But, we do not need this stronger version in this
paper.
\end{rema}


\section{Resolution of Derived Object}

\begin{say}\lab{Fitting}
Let $\sG$ be any coherent sheaf over a scheme $X$. A presentation
of  $\sG$ is an exact sequence
$$\sO_{X}^{\oplus n} \stackrel{\alpha}{\lra} \sO_{X}^{\oplus m}\lra \sG \lra 0.$$
The $h$-th $(h \le m)$ Fitting ideal of the above presentation,
denoted $\sJ_h(\sG)$,
 is the $(m-h)$-determinatal ideal of
the homomorphism $\alpha$ (and is defined to be $\sO_X$ when $h >
m$). The basic property of Fitting ideals is that any two
presentations of $\sG$ have the same
 Fitting ideals (\cite{Fitting, Northcott}).
 This enables us to define the Fitting ideals of the coherent sheaf $\sG$ without
 reference to any particular presentation. It can be shown that
 $\sJ_h(\sG)$ are finitely generated and form an increasing
 sequence
 $$\sJ_0(\sG) \subset \sJ_1(\sG) \subset \cdots.$$
Taking Fitting ideal commutes with base change. That is, if
 $g: Y \lra X$ is a morphism, then the $h$-th Fitting ideal of the
 sheaf $f^* \sG$ is generated, as an $\sO_Y$-module, by the $h$-th Fitting ideal of
 $\sG$. This is because the pullback of a presentation of $\sG$ is
 a presentation of $f^*\sG$ since the tensor product is a right
 exact functor.
 In addition, taking Fitting ideal commutes with localization.
 Let $S$ be any multiplicative closed subset of $\sO_X$ not containing the zero element. Then for every
 $h \ge 0$ we have
 $$\sJ_h(\sG)(\sO_X)_S = \sJ_h(\sG_S),$$
 where by $\sJ_h(\sG_S)$ we mean the $h$-th Fitting ideal of the sheaf $\sG_S$ of $(\sO_X)_S$-modules.
 For more details of Fitting ideals, the reader is referred to \cite{Fitting} and \cite{Northcott}.
\end{say}

\begin{say}\lab{notations} We now are ready to prove Theorem \ref{resolution} which we
restate below.  Since $\sEb$ is perfect and  can locally be represented
by a complex of locally free sheaves of finite length supported
only in non-negative degrees,  we can assume that $\2M$ admits an open cover $\coprod U$,
and over each open subset $U$, its restriction $\sEb|_U$ is
represented by the following complex of locally free sheaves
\beq\lab{localPre}
\begin{CD} \sF_0 @>{\psi_0}>>  \sF_1 @>{\psi_1}>> \cdots @>>> \sF_{n-1} @>{
\psi_{n-1}}>> \sF_n.
\end{CD}\eeq
\end{say}

\begin{theo}\lab{resolution-restated}
Let $\2M$ be an { integral}  DM stack and $\sEb$  a perfect
object in the derived category ${\rm D}^b(\2M)$ 
which can be locally represented
by a complex of locally free sheaves of finite length supported
only in non-negative degrees. 
Then there is another { integral} DM stack $\wfM$ and a
dominant birational morphism $f: \wfM\to \2M$ such that { over any open chart and} for every
$0 \le i \le n-1$, the homomorphism $\begin{CD}{ f}^*\sF_i
@>{{ f}^* \psi_i}>> { f}^*\sF_{i+1}\end{CD}$ is
diagonalizable. In particular, $\sH^0(Lf\sta \sEb)$ is locally
free.
\end{theo}
\begin{proof} We will adopt the notations from \ref{notations}
and assume that locally $\sEb$ is represented by the complex as in
\eqref{localPre}.

 We first consider the coherent
sheaf $\sH^n(\sFb)=\coker \psi_{n-1}$. The sequence
$$\begin{CD} \sF_{n-1} @>{ \psi_{n-1}}>> \sF_n @>>> \sH^n(\sFb) @>>>0
\end{CD}$$ is a locally free presentation of $\sH^n(\sEb)$. Since the
Fitting ideals of $\sH^n(\sEb)$ are the same as the determinatal
ideals of $ \psi_{n-1}: \sF_{n-1} \to \sF_n$ and Fitting ideals
are independent of presentations, we conclude that when applied to
the homomorphism $\psi_{n-1}: \sF_{n-1} \to \sF_n$ over the open
subset $U$, the iterated blowup as described in \S
\ref{determinantalBlowup} patch together to produce a well-defined
iterated blowup of $\2M$
$$f_{n-1}: \2M_{n-1} \lra \2M$$
such that $f_{n-1}^* \psi_{n-1}$  is { locally} diagonalizable.
{ Now observe that if $M$ is integral and $Z \subset M$ a
closed substack, then $L_Z M$, the blowup of $M$ along $Z$, is
also integral. This follows from the fact that if $I$ is any ideal
in a domain $A$, then $\bigoplus_n I^n$ is also a domain.} Thus
{ $\2M_{n-1}$ is integral and hence $\ker  f_{n-1}^*
\psi_{n-1}$ is locally free.}

Now we apply the left derived functor $Lf_{n-1}^*$ to $\sEb$
and obtain  $$Lf_{n-1}^* \sEb \in {\rm D}^b(M_{n-1}).$$ The
stack $M_{n-1}$ admits the open cover by $\coprod
f_{n-1}^{-1}(U)$. Over the open subset $ f_{n-1}^{-1}(U)$, $Lf_{n-1}^* \sEb$ is represented by
$$
\begin{CD}f_{n-1}^*\sF_0 @>{f_{n-1}^*\psi_0}>>  f_{n-1}^*\sF_1 @>>> \cdots f_{n-1}^*\sF_{n-1} @>{
f_{n-1}^*\psi_{n-1}}>> f_{n-1}^*\sF_n.
\end{CD}$$
Consider the short exact sequence
$$
\begin{CD}
f_{n-1}^*\sF_{n-2} @>{f_{n-1}^*\psi_{n-2}}>> \ker
f_{n-1}^*\psi_{n-1} @>>> \sH^{n-1} (Lf_{n-1}^*\sEb) @>>> 0.
\end{CD}$$
This is a locally free presentation of $\sH^{n-1} (Lf_{n-1}^*\sEb)$.  This allows us to apply the same blowing up
process as in the previous step to the coherent sheaf $\sH^{n-1}
(Lf_{n-1}^*\sEb)$ to get
$$f_{n-2} : \2M_{n-2} \lra \2M_{n-1}$$ such that $
f_{n-2}^*f_{n-1}^*\psi_{n-2}$ is diagonalizable and { for the
same reason as explained earlier} $\ker
f_{n-2}^*f_{n-1}^*\psi_{n-2}$ is locally free.

Applying  the above repeatedly (or by induction), we will
eventually arrive at a birational dominant morphism $$f: \wfM \lra
\2M,$$ factoring as
$$\begin{CD}
\wfM =M_0 @>{f_0}>> M_1  @>{f_1}>> \cdots  @>{f_{n-2}}>> M_{n-1}
@>{f_{n-1}}>> M,
\end{CD}
$$
  such that if we let $f$ to be $f_{n-1} \circ f_{n-2} \cdots
\circ f_0$, then for each $0 \le i \le n-1$, $f^* \psi_i $ is
diagonalizable, { $M_i$ is integral} and $\ker f^* \psi_i$ is
locally free. Now use that $Lf^*\sEb$ is locally represented by
$$
\begin{CD}f^*\sF_0 @>{f^*\psi_0}>>  f^*\sF_1 @>>> \cdots f^*\sF_{n-1} @>{
f^*\psi_{n-1}}>> f^*\sF_n,
\end{CD}$$
we see that $$\sH^0(Lf^*\sEb)=\ker f^*\psi_0,$$ the theorem is
thus proved.
\end{proof}

{
\begin{coro}\lab{2-term}
Let $\2M$ be a DM stack { (not necessarily integral)} and
$\sEb$ a  perfect object in the derived category ${\rm D}^b(\2M)$
with $\sH^i(\sEb)=0$ for $i<0$.  {  Assume further that locally
$\sEb$ can be represented by a two-term complex $\begin{CD} \sF_0
@>{ \psi}>> \sF_1\end{CD}$ of locally free sheaves.}
Then there is another  DM stack $\wfM$ and a  birational morphism
$f: \wfM\to \2M$ such that  the homomorphism $\begin{CD} f^*\sF_0
@>{ f^*\psi}>> f^*\sF_1\end{CD}$ is diagonalizable. In particular,
for any irreducible component $\wfM'$ of $\wfM$ endowed with the
reduced stack structure, $\sH^0(Lf\sta \sEb|_{\wfM'})$ is
locally free.
\end{coro}
}
\begin{proof} {
As in the first step of the proof of the above theorem,
we consider the locally free presentation
$$\begin{CD} \sF_0 @>{ \psi}>> \sF_1 @>>> \sH^1 (\Eb) @>>> 0\end{CD}$$
of the sheaf $\sH^1 (\Eb)$,} we then
simply blow up $M$ along its Fitting ideals and
obtain $f: \wfM\to \2M$. { It follows from Theorem
\ref{determ} that
 the homomorphism
$\begin{CD} f^*\sF_0 @>{ f^*\psi}>> f^*\sF_1\end{CD}$ is
diagonalizable.  Now Proposition \ref{usefulFacts} (2) implies the
rest of the statements. }
\end{proof}

As a direct consequence of Proposition \ref{prop:universality},
the stack $\wfM$ is universal.

\begin{prop}\lab{stack:universality} {\rm (Universality.)} { Let the
situation be as in Theorem \ref{resolution-restated} or as in
Corollary \ref{2-term}.}
 If ${ g}: Z \lra M$ is any dominant morphism such that for each $0
 \le i \le n-1$, ${ g}^*\psi: { g}^*\sF_i \lra { g}^*\sF_{i+1}$ is
diagonalizable, then ${ g}$ factors uniquely through ${ f}: \widetilde{M} \lra M$.
\end{prop}

But, for topological applications, the following base change
property, although weaker, is more convenient to use and easy to
prove. 

\begin{prop}\lab{baseChange} {\rm (Base Change Property.)}
{ Suppose that $M$ is integral.}
 For any  DM stack $N$ and a dominant morphism $g: N \to \2M$
 such that  $\sH^0(Lg\sta \sEb)$ is locally free,
 we can find another
 DM stack $N'$ and a dominant morphism $\tilde{g}: N'
\to \2M$, factoring through $f$ and $g$ (i.e., $\tilde{g}=f \circ
f'= g \circ g'$)
$$
\begin{CD}
N' @>{f'}>> \wfM \\
@V{g'}VV  @V{f}VV \\
 N @>{g}>> \2M
\end{CD}
$$
so that $\sH^0(L\tilde{g}{\sta} \sEb)$ is locally free and is
the pull back of $\sH^0(Lf\sta \sEb)$ and $\sH^0(Lg{\sta}
\sEb)$.
\end{prop}
\begin{proof}
Indeed, we let $N'$ be the closure of the open subset of the graph
of the rational map $$N \lra \wfM$$ that is isomorphic to its
image when projected to either $N'$ and $\wfM$. Then we obtain the
square as in the proposition such that $ \tilde{g}: N' \lra M$ is
a dominant  morphism. The rest conclusions are local. So, locally
we will represent $ Lf^* \sEb$  by the complex of locally free
sheaves
$$
\begin{CD} f^*\sF_0 @>{f^*\psi_0}>>  f^*\sF_1 @>>> \cdots.
\end{CD}$$
This implies that $L\tilde{g}^* \sEb = Lf'^*\circ Lf^*
\sEb$  is locally represented by
$$
\begin{CD} f'^*f^*\sF_0 @>{f'^*f^*\psi_0}>> f'^* f^*\sF_1 @>>> \cdots.
\end{CD}$$
Now, note that  $$
\begin{CD}0 @>>> \ker f^*\psi_0 @>>> f^*\sF_0  @>{f^*\psi_0}>>  f^*\sF_1\end{CD}$$
is exact. Since $f'$ is dominant,
 $$
\begin{CD}0 @>>> f'^*\ker f^*\psi_0 @>>> f'^*f^*\sF_0 @>{f'^*f^*\psi_0}>>  f'^*f^*\sF_1 \end{CD}$$
is also exact. Hence
$$\sH^0(L\tilde{g}{\sta} \sEb) = 
f'^* \sH^0(Lf^*\sEb).$$ Similarly,
$$\sH^0(L\tilde{g}{\sta} \sEb) = 
g'^* \sH^0(Lg^*\sEb).$$
\end{proof}

\begin{say} { Suppose that $M$ is integral}.  One can also routinely verify the following.
\begin{enumerate}
\item Let $\sEb$ be a perfect derived object over an DM stack
$\2M$ and
 ${\sEb}'$ be a perfect derived object over an DM stack $\2M'$, both with vanishing
 $\sH^{i<0}$.
 Let $f: \wfM \lra \2M$ and  $f': \wfM' \lra \2M'$ be given as in
 Theorem \ref{resolution-restated}, then $\sH^0(L(f,f')^* (\sEb
 \boxplus {\sEb}'))$ is locally free over $\wfM \times \wfM'$;
 \item Let $\sEb$ and ${\sEb}'$ be perfect derived objects over a DM stack
$\2M$, both with vanishing
 $\sH^{i<0}$.  Let $f: \wfM \lra \2M$ and  $f': \wfM' \lra \2M$ be
given as in Theorem \ref{resolution-restated}. We let ${\bf M}$ be
the graph of the rational map $\wfM \lra \wfM'$ and $\tilde{f}$
the projection to $\2M$. Then $\sH^0(L\tilde{f}^* (\sEb
 \oplus {\sEb}'))$ is locally free over ${\bf M}$.
\end{enumerate}
\end{say}

\section{The Euler class of Perfect Derived Object}\lab{Euler}

\begin{say}
{ In this subsection, we again make assumption that $M$
(hence also $\wfM$) is integral.} To define the Euler class of the
complex $\sEb$, { we} will use the top Chern class of the locally
free sheaf $\sH^0(Lf\sta \sEb)$. Let $r=\rank \sH^0(\sEb)$. The
Chern class $c_r(\sH^0(Lf\sta \sEb))$ is a homomorphism
$$\begin{CD}c_r(\sH^0(Lf\sta \sEb)): A_*(\wfM) @>>> A_{*-r}(\wfM)\end{CD}.$$
Here $A_* (\wfM)$ is the Chow group of cycles on  $\wfM$. We will
assume that $\rank \sH^0 (\sEb) >0$ and the higher cohomology
sheaves  $\sH^i (\sEb)$ are all torsion for $i >0$. This way, the
Euler class of the complex $\sEb$, as expected, should only
{ depend} on $\sH^0(Lf\sta \sEb)$.
\end{say}

\begin{defi}
 For { any integral}  DM stack $\2M$ of dimension $n$ and { a} derived object
$\sEb$ as in Theorem \ref{resolution-restated}, let $r=\rank
\sH^0(\sEb).$ Then we  define its Euler class ${ \ee}(\sEb) \in A_{n-r}
(\2M)$ { as}: \beq { \ee}(\sEb) := f_*(c_r(\sH^0(Lf\sta \sEb)) \cdot
[\wfM]). \eeq Here
 $A_* (\2M)$ is the Chow group of cycles on  $\2M$.
\end{defi}


\begin{prop}\lab{independent} For any { integral} DM stack $N$ and a  surjective birational morphism $g:
N \to \2M$ such that $\sH^0(Lg\sta \sEb)$ is locally free, we
have
$$g_*(c_r(\sH^0(Lg\sta \sEb)) \cdot [N])=f_*(c_r(\sH^0(Lf\sta \sEb))\cdot [\wfM]).$$
\end{prop}
\begin{proof}
By Proposition \ref{baseChange}, we have a square
$$
\begin{CD}
N' @>{f'}>> \wfM \\
@V{g'}VV  @V{f}VV \\
 N @>{g}>> \2M.
\end{CD}
$$
The diagonal morphism $N' \lra \2M$, denoted $\tilde{g}$, is a surjective birational morphism.
From Proposition \ref{baseChange}, we have that
$$\sH^0(L\tilde{g}{\sta} \sEb)= f'^* \sH^0(Lf^*\sEb)$$ (note that all the
sheaves involved are locally free). Observe that we also have
 $f'_*[N'] =[\wfM]$ because $f'$ is birational and surjective. Hence
 we obtain
$$f'_*( c_r(\sH^0(L\tilde{g}\sta \sEb)) \cdot [N']) = f'_*(f'^*(c_r(\sH^0(Lf\sta \sEb)))\cdot [N'])$$
$$= c_r(\sH^0(Lf\sta \sEb)) \cdot [\wfM].$$
This implies that
$$f_* (c_r(\sH^0(Lf\sta \sEb)) \cdot [\wfM]) =  f_*f'_*( c_r(\sH^0(L\tilde{g}\sta \sEb)) \cdot [N'])$$
$$= \tilde{g}_*(c_r(\sH^0(L\tilde{g}\sta \sEb)) \cdot [N']).$$
Similarly, we get
 $$g_* (c_r(\sH^0(Lg\sta \sEb)) \cdot [N]) =  g_*g'_*( c_r(\sH^0(L\tilde{g}\sta \sEb)) \cdot [N'])$$
$$= \tilde{g}_*(c_r(\sH^0(L\tilde{g}\sta \sEb)) \cdot [N']).$$
This proves the proposition.
\end{proof}

\begin{coro}\lab{Euler} Let $\sEb$ be a perfect object in ${\rm D}^b(M)$.
Suppose $\sH^{i<0}(\sEb)=0$  and
 $\sH^{i>0}(\sEb)$ are torsion.  Then the Euler class
${ \ee}(\sEb)$ is well-defined and { is} independent of the choice of the
bitational surjective morphisms ${ f}: \wfM \lra M$.
\end{coro}

\section{Applications to GW-Invariants}

\def\fMPdgn{\overline{\mathfrak M}_g(\mathbf P,d)}
\def\fXg{{\fMPdgn}}

\begin{say}
Let $\bP$ be a projective space, and let
$\fMPdgn$ be the DM stack of degree $d$ genus $g$ stable maps to $\bP$
as before.
Let \beq(\ff,\pi): \cX\lra \bP\times\fXg \eeq be its
universal family. For any positive integer $k$,  the
derived object
${R} \pi_* (\ff^*\sO_{{{\bP}}}(k))$
in ${ {\rm D}^b}(\fXg)$ is perfect.
\end{say}

\begin{say} One way to see this is to pick two sufficiently large integers $n$ and $n'$
and form $\sL= \ff^*\sO_{\bP}(n) \otimes {
\omega_{\cX/\fXg}^{\otimes n'}}$; then form the tautological homomorphism
$$\cA_0=\pi^* \pi_*
(\ff^*\sO_{\bP}(k) \otimes \sL) \otimes \sL^{-1}\lra \ff\sta\sO_{\bP}(k).
$$
Since $n$ and $n'$ are sufficiently large, it is surjective. We let $\cA_{-1}$ be the kernel of the above
homomorphism.
Then it is easy to see that we have an quasi-isomorphism
$$[R^1\pi\lsta \cA_{-1}\to R^1\pi\lsta \cA_0]={R} \pi_* \ff\sta
\sO_{\bP}(k).
$$
Again since $n$ and $n'$ are sufficiently large,
$$[\sE_0\to\sE_1]\defeq [R^1\pi\lsta \cA_{-1}\to R^1\pi\lsta \cA_0]$$
is a complex of locally free sheaves. This proves that
${R} \pi_*\ff\sta\sO_{\bP}(k)$ is perfect.
\end{say}

{
\begin{defi}\lab{prime-comp}
 Assume $d > 2g-2$. We let $\fM_g({\bf P},d)_0\sub\fXg$ be the
open subset consisting of stable morphisms with irreducible domain curves.
We define the {\sl primary} part $\fXg'$ of $\fXg$ to be the closure
of $\fM_g({\bf P},d)_0$ in $\fXg$.
\end{defi}

The open subset   $\fM_g({\bf P},d)_0$ is non-empty, smooth and has the expected dimension.
Thus  $\fXg'$ is generically smooth and of the expected dimension.

\begin{rema}\lab{primaryComp}  Some remarks on the primary components are in order.
Let $C$ be an irreducible curve of genus $g$. A map $u: C \lra \bP$ is given by
$(m+1)$-sections $u_0, \cdots, u_m \in \Gamma(u^*\sO_{\bP}(1))$, where $m=\dim\bP$. We may assume that
$u^*\sO_{\bP}(1)=\sO_C(D)$ for some effective divisor $D$. Assume that $d >g$. Then there are general divisors
on the curve $C$. If $D$ is general, by the geometric version
of Riemann-Roch theorem, $\dim \Ga(\sO_C(D))= d+1-g \ge 2$. From here, one checks that the dimension of
$\fXg$ at such a map is
$$3g-3+ d + (d+1-g)m =d(m+1) + (m-3)(1-g),$$ as expected.
When $d > 2g-2$, all divisors are general. If $C$ is reducible and has more than one irreducible components
(of positive genera, for instance) that are not contracted by the stable morphsim, then conjecturally they do not contribute the GW number of quintic Calabi-Yaus.
When $g < d < 2g-1$, there are special divisors over the curve $C$. Such divisors give rise to $\Ga(\sO_C(D))$ with
$\dim \Ga(\sO_C(D))>d+1-g \ge 2$. Hence they may produce a component of $\fXg$ with dimension larger than expected.
When $0<d\le g$, it may happen that none of the irreducible components of $\fXg$ have the expected dimension.
\end{rema}

\begin{say}\lab{1impliesk}\lab{modularBlowup}
We now assume $d>2g-2$.
We apply the constructions in the previous sections to  the complex $R \pi_* \ff^* \sO_{{{\bP}}}(k)$
restricted to $\fXg'$. Let $\ff'$ be the restriction of $\ff$ to $\fXg'$. By the vanishing of high cohomology,
$R^1 \pi_*\ff^{\prime \ast} \sO_{\bP}(k))$
is trivial on a dense open subset of $\fXg'$. Then
by \S \ref{Euler}, we have a well-defined {\it modular Euler class}
\beq\lab{ModularEulerClass}\ee(R \pi_*( \ff^{\prime\ast}\sO_{{{\bP}}}(k))) \in A_* (\fXg').\eeq
We denote this number by $N_{g,d}'$, (cf. Definition \ref{prime-comp}).
\end{say}

\begin{say} When $g=0$, $N_{0,d}'=N_{0,d}$, the usual Gromov-Witten
number of the quintic $X$; when $g=1$, $N_{1,d}'$  is the
reduced genus-1 GW-invariants (see \cite{LZ, CL}).
\end{say}

We believe these numbers are the reduced GW-invarains of guintics
conjectured by Li-Zinger, (cf. Conj. \ref{conj}).

}


\begin{thebibliography}{10}
\bibitem{CL} H.-L. Chang and J. Li, {\em On reduced genus one GW-invariants of Quintics},
 in preparation.
\bibitem{Fitting} H. Fitting, {\em Die Determinantenideale eines
Moduls,} Jahresber Deutsch. Math. Verein. 46 (1936), 195-228.
\bibitem{Fulton} W. Fulton, {\em Intersection Theory.} Springer,
1984.
\bibitem{Hartshorne} R. Hartshorne,
{\em Algebraic Geometry.}  GTM 52 (1977), Springer-Verlag.
\bibitem{HL08} Yi Hu and Jun Li, {\em Genus-One Stable Maps, Local
Equations and Vakil-Zinger's desingularization.}
Math. Ann. (2010)
\bibitem{HLS} Yi Hu, Jiayuan Lin and Yijun Shao,
{\em A compactification of the space of algebraic maps from $\Po$
to $\Pn$,} math.AG/0701255.
\bibitem{LZ} J. Li and A. Zinger,
{\em On the Genus-One Gromov-Witten Invariants of Complete
Intersections.} J. of Differential Geom. 82 (2009), no 3, 641-690.
\bibitem{Northcott} D. G. Northcott,
{\em Finite free resolutions.} Cambridge Tracts in Mathematics,
No. 71. Cambridge University Press, Cambridge-New York-Melbourne,
1976.
\bibitem{VZ} R. Vakil and A. Zinger,
{\em A Desingularization of the Main Component of the Moduli Space
of Genus-One Stable Maps into $\Pn$.} Geom. Topol.  12  (2008),
no. 1, 1--95.
\end{thebibliography}
\end{document}